\documentclass[12pt, reqno]{amsart}

\author[P.~Leonetti]{Paolo Leonetti}
\address{Universit\`a ``Luigi Bocconi''\\Department of Statistics\\Milan, Italy}
\email{leonetti.paolo@gmail.com}
\urladdr{\url{https://sites.google.com/site/leonettipaolo/}}


\title{Invariance of Ideal Limit Points}

\usepackage{amsmath}
\usepackage{amssymb}
\usepackage{amsthm}
\usepackage[left=3.5cm, right=3.5cm, paperheight=11.8in]{geometry}
\usepackage{hyperref}
\usepackage{fancyhdr}
\usepackage[inline]{enumitem}
\usepackage{comment}
\usepackage{nicefrac}
\usepackage{mathrsfs}
\usepackage{graphicx}
\usepackage[utf8]{inputenc}

\AtBeginDocument{%
   \def\MR#1{}
}

\newtheorem{thm}{Theorem}[section]
\newtheorem{cor}[thm]{Corollary}
\newtheorem{lem}[thm]{Lemma}
\newtheorem{prop}[thm]{Proposition}

\theoremstyle{definition} 
\newtheorem{defi}[thm]{Definition}
\let\olddefi\defi
\renewcommand{\defi}{\olddefi\normalfont}
\newtheorem{example}[thm]{Example}
\let\oldexample\example
\renewcommand{\example}{\oldexample\normalfont}
\newtheorem{rmk}[thm]{Remark}
\let\oldrmk\rmk
\renewcommand{\rmk}{\oldrmk\normalfont}

\hypersetup{
    pdftitle={Invariance of Ideal Limit Points},
    pdfauthor={Paolo Leonetti},
    pdfmenubar=false,
    pdffitwindow=true,
    pdfstartview=FitH,
    colorlinks=true,
    linkcolor=blue,
    citecolor=green,
    urlcolor=cyan
}

\providecommand{\MR}[1]{}

\providecommand{\MR}{\relax\ifhmode\unskip\space\fi MR }

\providecommand{\href}[2]{#2}
                                 
\begin{document}

\maketitle
\thispagestyle{empty}

\begin{abstract}
\noindent 
Let $\mathcal{I}$ be an analytic P-ideal [respectively, a summable ideal] on the positive integers and let $(x_n)$ be a sequence taking values in a metric space $X$. First, it is shown that the set of ideal limit points of $(x_n)$ is an $F_\sigma$-set [resp., a closet set]. 

Let us assume that $X$ is also separable and the ideal $\mathcal{I}$ satisfies certain additional assumptions, which however includes several well-known examples, e.g., the collection of sets with zero asymptotic density, sets with zero logarithmic density, and some summable ideals. Then, it is shown that the set of ideal limit points of $(x_n)$ is equal to the set of ideal limit points of almost all its subsequences.
\end{abstract}

\section{Introduction}\label{sec:intro}

Let $\mathcal{I}$ be an ideal on the positive integers $\mathbf{N}$, i.e., a collection of subsets of $\mathbf{N}$ closed under taking finite unions and subsets. It is assumed that $\mathcal{I}$ contains all finite subsets of $\mathbf{N}$ and it is different from the whole power set $\mathcal{P}(\mathbf{N})$. 
Note that the family of subsets with zero asymptotic density
$$
\mathcal{I}_0:=\left\{S\subseteq \mathbf{N}: |S \cap [1,n]|=o(n) \,\text{ as }n\to \infty\right\}
$$
is an ideal, cf. Section \ref{sec:preliminaries}. 

Let $X$ be a topological space, which will be always assumed to be Hausdorff. Given an $X$-valued sequence $x=(x_n)$, we 
denote by $\Lambda_x(\mathcal{I})$ the set of $\mathcal{I}$\emph{-limit points} of $x$, that is, the set of all $\ell \in X$ such that 
$$
\lim_{k\to \infty}x_{n_k}=\ell,
$$
for some subsequence $(x_{n_k})$ such that $\{n_k: k \in \mathbf{N}\} \notin \mathcal{I}$. Statistical limit points (i.e., $\mathcal{I}_{0}$-limit points) of real sequences were introduced by Fridy \cite{MR1181163}, cf. also \cite{MR1372186, MR2463821, MR1416085, MR1838788} and references therein.

An old result of Buck \cite{MR0009997} states that the set of ordinary limit points of ``almost every'' subsequence of a real sequence $(x_n)$ coincides with the set of ordinary limit points of the original sequence, in the sense of Lebesgue measure. The aim of this article is to extend this result to the setting of ideal limit points. 

Its analogue for ideal cluster points has been recently obtained in \cite{Leo17} (where we recall that $\ell \in X$ is an $\mathcal{I}$-cluster point of $x$ provided that $\{n: x_n \in U\} \notin \mathcal{I}$ for all neighborhoods $U$ of $\ell$). Related results can be found in \cite{MR3568092, MR0316930, MR1260176, MR1924673}.


\section{Preliminaries}\label{sec:preliminaries}

We recall that an ideal $\mathcal{I}$ is said to be a \emph{P-ideal} if for every sequence $(A_n)$ of sets in $\mathcal{I}$ there exists $A \in \mathcal{I}$ such that $A_n\setminus A$ is finite for all $n$; equivalent definitions were given, e.g., in \cite[Proposition 1]{MR2285579}.

By identifying sets of integers with their characteristic function, we equip 
$\mathcal{P}(\mathbf{N})$ with the Cantor-space topology and therefore we can assign the topological complexity to the ideals on $\mathbf{N}$. In particular, an ideal $\mathcal{I}$ is \emph{analytic} if it is a continuous image of a $G_\delta$-subset of the Cantor space. 
Moreover, a map $\varphi: \mathcal{P}(\mathbf{N}) \to [0,\infty]$ is a \emph{lower semicontinuous submeasure} provided that:
  (i) $\varphi(\emptyset)=0$; 
  (ii) $\varphi(\{n\}) <\infty$ for all $n\in \mathbf{N}$; 
  (iii) $\varphi(A) \le \varphi(B)$ whenever $A\subseteq B$;
  (iv) $\varphi(A\cup B) \le \varphi(A)+\varphi(B)$ for all $A,B$; and
  (v) $\varphi(A)=\lim_{n}\varphi(A\cap \{1,\ldots,n\})$ for all $A$.

By a classical result of Solecki, an ideal $\mathcal{I}$ is an analytic P-ideal if and only if there exists a lower semicontinuous submeasure $\varphi$ such that 
\begin{equation}\label{eq:analyticPideal}
\mathcal{I}=\mathcal{I}_\varphi:=\{A\subseteq \mathbf{N}: \|A\|_\varphi =0\},
\end{equation}
where $\|A\|_\varphi:=\lim_n \varphi(A\setminus \{1,\ldots,n\})$ for all $A\subseteq \mathbf{N}$, cf e.g. \cite[Theorem 1.2.5]{MR1711328}. Hereafter, unless otherwise stated, an analytic P-ideal will be always denoted by $\mathcal{I}_\varphi$, where $\varphi$ stands for the associated lower semicontinuous submeasure as in \eqref{eq:analyticPideal}. 

Lastly, given $k \in \mathbf{N}$ and \emph{infinite} sets $A,B\subseteq \mathbf{N}$ with canonical enumeration $\{a_n: n \in \mathbf{N}\}$ and $\{b_n: n \in \mathbf{N}\}$, respectively, we write $A\le B$ if $a_n\le b_n$ for all $n \in \mathbf{N}$ and define
$$
A_B:=\{a_b: b \in B\} \,\,\text{ and }\,\,kA:=\{ka: a \in A\}.
$$

At this point, we recall the definition of thinnability given in \cite[Definition 2.1]{Leo17}.
\begin{defi}\label{def:thinn0}
An ideal $\mathcal{I}$ is said to be \emph{weakly thinnable} if $A_B \notin \mathcal{I}$ whenever $A\subseteq \mathbf{N}$ admits non-zero asymptotic density and $B\notin \mathcal{I}$.

If, in addition, also $B_A\notin \mathcal{I}$ and $X\notin \mathcal{I}$ whenever $X \le Y$ and $Y\notin \mathcal{I}$, then $\mathcal{I}$ is said to be \emph{thinnable}.

As it has been shown in \cite[Proposition 2.3]{Leo17}, the class of thinnable ideals are quite rich and include well-known examples, e.g., the collection of sets with zero asymptotic density, sets with zero logarithmic density, and some summable ideals. 
Moreover, in the special case of analytic P-ideals, we define also strong thinnability:

\end{defi}
\begin{defi}\label{def:strongthinn}
An analytic P-ideal $\mathcal{I}_\varphi$ is said to be \emph{strongly thinnable} if: 
\begin{enumerate}[label=(\roman*)]
\item \label{item1} $\mathcal{I}_\varphi$ is weakly thinnable; 
\item \label{item2} given $q>0$ and a set $A\subseteq \mathbf{N}$ with asymptotic density $a >0$, there exists $c=c(q,a)>0$ such that $\|B_A\|_\varphi \ge cq$ whenever $\|B\|_\varphi\ge q$;
\item \label{item3} there exists $c>0$ such that $\|X\|_\varphi\ge c\|Y\|_\varphi$ whenever $X \le Y$. 
\end{enumerate}
\end{defi}

A moment thought reveals that strongly thinnability is just a refinement of thinnability, considering that $\|\cdot\|_\varphi$ allows us to quantify the ``largeness'' of subsets of $\mathbf{N}$.

\begin{prop}\label{thm:thinnableold}
Let $f: \mathbf{N} \to (0,\infty)$ be a definitively non-increasing function such that $\sum_{n\ge 1}f(n)=\infty$. In addition, suppose that 
\begin{equation}\label{eq:addlimit}
\liminf_{n\to \infty}\frac{\sum_{i \in [1,n]}f(i)}{\sum_{i \in [1,kn]}f(i)}\neq 0\,\,\,\text{ for all }k \in \mathbf{N}
\end{equation} 
and define the ideal
$$
\mathscr{E}_f:=\left\{S\subseteq \mathbf{N}: \lim_{n\to \infty}\frac{\sum_{i \in S\cap [1,n]}f(i)}{\sum_{i \in [1,n]}f(i)}=0\right\}.
$$
Then, $\mathscr{E}_f$ is a strongly thinnable analytic P-ideal 
provided that $\mathscr{E}_f$ is strechable, i.e., $kA\notin \mathscr{E}_f$ for all $k \in \mathbf{N}$ and $A\notin \mathscr{E}_f$.
\end{prop}
\begin{proof}
Note that $\mathscr{E}_f$ is a Erd{\H o}s--Ulam ideal, indeed $f(n)=o(f(1)+\cdots+f(n))$ as $n\to \infty$ since $f$ is non-increasing, cf. \cite[Section 1.13]{MR1711328}; hence $\mathscr{E}_f$ contains all the finite sets. Moreover, the weak thinnability of $\mathscr{E}_f$, i.e., property \ref{item1}, has been shown in \cite[Proposition 2.3]{Leo17} (its proof relies on the hypotheses \eqref{eq:addlimit} and strechability).

Let $\varphi$ be a lower semicontinuous submeasure associated with $\mathscr{E}_f$. Then, it follows from the proof of \cite[Theorem 1.13.3]{MR1711328} that there exists a strictly increasing sequence of positive integers $(z_n)$ such that 
\begin{equation}\label{eq:conditionzk}
\lim_{n\to \infty}\frac{\sum_{s \in (z_n,z_{n+1}]}f(s)}{\sum_{s \in [1,z_n]}f(s)}=1
\end{equation}
and $\|S\|_\varphi=\lim_{n\to \infty}g_n(S)$ for all $S\subseteq \mathbf{N}$, where
$$
g_n(S):=\sup_{k\in \mathbf{N}}\,\frac{\sum_{s \in S \cap (z_{k},z_{k+1}] \setminus \{1,\ldots,n\}}f(s)}{\sum_{s \in [1,z_k]}f(s)}.
$$
Considering that $g_n(S) \downarrow \|S\|_\varphi$, then also $g_{z_n}(S)\downarrow \|S\|_\varphi$. Hence
\begin{equation}\label{eq:varphidensity}
\begin{split}
\|S\|_\varphi&=\inf_{n\in \mathbf{N}}\,g_{z_n}(S)=\inf_{n\in \mathbf{N}}\,\sup_{k\in \mathbf{N}}\,\frac{\sum_{s \in S \cap (z_{k},z_{k+1}] \setminus \{1,\ldots,z_n\}}f(s)}{\sum_{s \in [1,z_k]}f(s)}\\
&=\inf_{n\in \mathbf{N}}\,\sup_{k\ge n}\,\frac{\sum_{s \in S \cap (z_{k},z_{k+1}]}f(s)}{\sum_{s \in [1,z_k]}f(s)}=\limsup_{n\to \infty}\,\frac{\sum_{s \in S \cap (z_{n},z_{n+1}]}f(s)}{\sum_{s \in [1,z_n]}f(s)}.
\end{split}
\end{equation}
Replacing $\varphi$ with $\frac{1}{2}\varphi$ (which is possible since $\mathcal{I}_\varphi=\mathcal{I}_{\frac{1}{2}\varphi}$), we obtain by \eqref{eq:conditionzk} and \eqref{eq:varphidensity} that
\begin{equation}\label{eq:Svarphifinal}
\|S\|_\varphi=\limsup_{n\to \infty}\,\frac{\sum_{s \in S \cap (z_{n},z_{n+1}]}f(s)}{\sum_{s \in [1,z_{n+1}]}f(s)}.
\end{equation}

At this point, fix a set $A\subseteq \mathbf{N}$ with canonical enumeration $\{a_n: n \in \mathbf{N}\}$ such that $A$ admits asymptotic density $a \in (0,1]$. Fix also a real $q>0$ and a set $B\subseteq \mathbf{N}$ with canonical enumeration $\{b_n: n \in \mathbf{N}\}$ such that $\|B\|_\varphi \ge q$. Set $r:=\lfloor 1/a\rfloor +1$. Then, it follows by \eqref{eq:Svarphifinal} that
\begin{displaymath}
\begin{split}
\|B_A\|_\varphi&=\limsup_{n\to \infty}\,\frac{\sum_{z_{n}<b_{a_k}\le z_{n+1}}f(b_{a_k})}{\sum_{s \in [1,z_{n+1}]}f(s)}\\
&\ge \limsup_{n\to \infty}\,\frac{O(1)+\sum_{z_{n}<b_{a_k}\le z_{n+1}}f(b_{rk})}{\sum_{s \in [1,z_{n+1}]}f(s)}
\end{split}
\end{displaymath}
Hence, considering that by \eqref{eq:conditionzk} it holds $z_{n+1}-z_{n} \ge z_n \ge n \to \infty$ and 
\begin{displaymath}
\begin{split}
\sum_{\substack{s \in S \cap (z_{n},z_{n+1}], \\ s\equiv 0\bmod{r}}}f(s) &\ge O(1)+\sum_{\substack{s \in S \cap (z_{n},z_{n+1}], \\ s\equiv 1\bmod{r}}}f(s) \\
&\ge \cdots \ge O(1)+\sum_{\substack{s \in S \cap (z_{n},z_{n+1}], \\ s\equiv r-1\bmod{r}}}f(s) \ge O(1)+\sum_{\substack{s \in S \cap (z_{n},z_{n+1}], \\ s\equiv 0\bmod{r}}}f(s)
\end{split}
\end{displaymath}
for every $S \subseteq \mathbf{N}$, we obtain that
$$
\|B_A\|_\varphi \ge \limsup_{n\to \infty}\,\frac{\sum_{z_{n}<b_{a_k}\le z_{n+1}}f(b_{rk})}{\sum_{s \in [1,z_{n+1}]}f(s)} \ge \frac{\|B\|_\varphi}{r} \ge \frac{q}{r},
$$
which proves property \ref{item2}.

Finally, fix sets $X,Y\subseteq \mathbf{N}$ with $X\le Y$ and define
$$
h_n(X):=\frac{\sum_{s \in X \cap (z_n,z_{n+1}]}f(s)}{\sum_{s \in [1,z_{n+1}]}f(s)}
$$
for each $n \in \mathbf{N}$, and similarly $h_n(Y)$. Since $\|Y\|_\varphi=\limsup_{n\to \infty}h_n(Y)$ by \eqref{eq:Svarphifinal}, it follows that there exists an infinite set $\mathcal{N}$ such that 
$
h_n(Y) \ge \frac{1}{2}\|Y\|_\varphi
$ 
for all $n \in \mathcal{N}$. Set also $\mu_n:=\sum_{s \in [1,z_{n+1}]}f(s)$ for each $n$. Note that the limit \eqref{eq:conditionzk} implies that $\mu_n \le \frac{2}{3}\mu_{n+1}$ for all sufficiently large $n$. Hence,
considering that $X\le Y$, we obtain
\begin{displaymath}
\begin{split}
h_n(Y)\mu_n&=\sum_{s \in Y \cap (z_n,z_{n+1}]}f(s) \le \sum_{s \in X \cap [1,z_{n+1}]}f(s)= \sum_{s \in X\cap [1,z_{1}]}f(s)+\sum_{i=1}^n h_i(X)\mu_i	\\
&\le z_1f(1)+\sum_{i=1}^n h_i(X)\mu_i \le O(1)+\mu_n\sum_{i=1}^n \left(\frac{2}{3}\right)^{n-i} h_i(X)
\end{split}
\end{displaymath}
for each $n \in \mathcal{N}$. Since $\mu_n \to \infty$ by hypothesis, it follows that 
$
h_n(Y) \le o(1)+\sum_{i=1}^n \left(\frac{2}{3}\right)^{n-i} h_i(X)
$ 
whenever $n \in \mathcal{N}$ is sufficiently large. Then 
$$
\|X\|_\varphi=\limsup_{n\to \infty} h_n(X) \ge \frac{1}{6}\|Y\|_\varphi;
$$
indeed, in the opposite, we would get
\begin{displaymath}
\begin{split}
\frac{1}{2}\|Y\|_\varphi \le h_n(Y) & \le o(1)+\sum_{i=1}^n \left(\frac{2}{3}\right)^{n-i} h_i(X)\\
&\le o(1)+\frac{1}{5}\|Y\|_\varphi\sum_{i=1}^n \left(\frac{2}{3}\right)^{n-i}<\frac{1}{2}\|Y\|_\varphi
\end{split}
\end{displaymath}
for each sufficiently large $n \in \mathcal{N}$. This proves property \ref{item3}, concluding the proof.
\end{proof}

For each real parameter $\alpha \ge -1$, let 
\begin{equation}\label{eq:Ialpha}
\mathcal{I}_\alpha:=\{S\subseteq \mathbf{N}: \mathrm{d}^\star_\alpha\,(S)=0\}
\end{equation}
be the ideal of subsets of zero $\alpha$-density, where
$$
\mathrm{d}_\alpha^\star: \mathcal{P}(\mathbf{N}) \to \mathbf{R}: S\mapsto \limsup_{n\to \infty}\frac{\sum_{i \in S\cap [1,n]} i^\alpha}{\sum_{i \in [1,n]} i^\alpha}
$$
denotes the upper $\alpha$-density on $\mathbf{N}$. Note that $\mathcal{I}_\alpha$ is a Erd{\H o}s--Ulam ideal. 

Recalling that every Erd{\H o}s--Ulam ideal is an analytic P-ideal, see e.g. 
\cite[Example 1.2.3.(d)]{MR1711328}, 
the following is immediate by Proposition \ref{thm:thinnableold} (we omit details):
\begin{cor}\label{cor:Ialpha}
$\mathcal{I}_\alpha$ is a strongly thinnable analytic P-ideal whenever $\alpha \in [-1,0]$.
\end{cor}

%


\section{Topological structure}\label{sec:top}

Our first result about the topological structure of ideal limit points sets follows: 


\begin{thm}\label{thm:Fsigma}
Let $x=(x_n)$ be a sequence taking values in a metric space $X$ and let $\mathcal{I}_\varphi$ be an analytic P-ideal. Then, the set
\begin{displaymath}
\Lambda_x(\mathcal{I}_\varphi,q):=\left\{\ell \in X: \lim_{n\to \infty,\, n \in A} x_{n} = \ell \text{ for some }A\subseteq \mathbf{N} \text{ such that }\|A\|_\varphi \ge q\right\}
\end{displaymath}
is closed for each $q>0$. In particular, $\Lambda_x(\mathcal{I}_\varphi)$ is an $F_\sigma$-set.
\end{thm}
\begin{proof}
Fix $q>0$. The claim is clear if $\Lambda_x(\mathcal{I}_\varphi,q)$ is empty. Hence, let us suppose hereafter that $\Lambda_x(\mathcal{I}_\varphi,q)\neq \emptyset$. Let $(\ell_m)$ be a sequence of limit points in $\Lambda_x(\mathcal{I}_\varphi,q)$ 
such that $\lim_m \ell_m=\ell$. By hypothesis, for each $m$ there exists a set $A_m \subseteq \mathbf{N}$ such that $\lim_{n\to \infty,\, n \in A_m} x_{n} = \ell_m$ and 
$$
\|A_m\|_\varphi = \lim_{n\to \infty}\varphi(A_m\setminus \{1,\ldots,n\})=\inf_{n \in \mathbf{N}} \varphi(A_m\setminus \{1,\ldots,n\}) \ge q.
$$

At this point, let $d$ denote the metric on $X$ and 
define 
\begin{equation}\label{eq:bm}
B_m:=\left\{n\in A_m: d\left(\ell_m,x_n\right) \le \frac{1}{m}\right\}.
\end{equation}
Note that, by construction, each $A_m\setminus B_m$ is finite. Set for convenience $\theta_0:=0$ and define recursively the increasing sequence of positive integers $(\theta_m: m \in \mathbf{N})$ so that $\theta_m$ is the smallest integer greater than both $\theta_{m-1}$ and $\max(A_{m+1}\setminus B_{m+1})$ such that
$$
\varphi(A_m \cap (\theta_{m-1},\theta_m]) \ge q-\nicefrac{1}{m}.
$$
(Note that $\theta_m$ is well defined because $\varphi$ is lower semicontinuous.)

Finally, define 
$
A:=\bigcup_{m\in \mathbf{N}} \left(A_m \cap (\theta_{m-1},\theta_m]\right).
$ 
Let us verify that $A\notin \mathcal{I}_\varphi$ and that the subsequence $(x_n: n \in A)$ converges to $\ell$. 
On the one hand, since $\theta_n \ge n$, we obtain
$$
\varphi(A\setminus \{1,\ldots,n\}) \ge \varphi(A_m \cap (\theta_{m-1},\theta_m]) \ge q-\nicefrac{1}{m}
$$
whenever $m \ge n+1$, hence $\|A\|_\varphi=\inf_{n \in \mathbf{N}}\varphi(A\setminus \{1,\ldots,n\}) \ge q$. 
On the other hand, fix $\varepsilon>0$. Then, there exists $m_0=m_0(\varepsilon) \in \mathbf{N}$ such that 
\begin{equation}\label{eq:m0}
m_0\ge \nicefrac{2}{\varepsilon}\,\,\text{ and }\,\,d(\ell_m,\ell) \le \nicefrac{\varepsilon}{2}
\end{equation}
whenever $m\ge m_0$. It follows that
\begin{equation}\label{eq:m1}
d(x_n,\ell) \le d(x_n,\ell_m)+d(\ell_m,\ell) \le \frac{1}{m}+\frac{\varepsilon}{2} 
\le \varepsilon
\end{equation}
for all $n \in A_m \cap (\theta_{m-1},\theta_m]$ and $m\ge m_0$. 
We conclude by the arbitrariness of $\varepsilon$ that 
$
\lim_{n\to \infty,\, n \in A} x_{n} = \ell.
$ 
In particular, $\Lambda_x(\mathcal{I}_\varphi)=\bigcup_{0< q\text{ rational}}\Lambda_x(\mathcal{I}_\varphi,q)$ is an $F_\sigma$-set.
\end{proof}

It is worth noting that Theorem \ref{thm:Fsigma} generalizes \cite[Theorem 2.6]{MR2463821} and \cite[Theorem 1.1]{MR1838788} for the case $\mathcal{I}_\varphi$ equal to the ideal $\mathcal{I}_0$; in addition, the result essentially appears also in \cite[Theorem 2]{MR2923430}. However, all these proofs seem to be incomplete as it is not clear why the constructed subsequence $(x_n: n \in A)$ converges to $\ell$.

The following corollary is immediate:
\begin{cor}\label{cor:erdosulam}
Let $x$ be a sequence taking values in a metric space and let $\mathcal{I}_\varphi$ be a Erd{\H o}s--Ulam ideal. Then, $\Lambda_x(\mathcal{I}_\varphi)$ is an $F_\sigma$-set.
\end{cor}

As it is shown in the following example, it may be the case that $\Lambda_x(\mathcal{I}_\varphi)$ is not closed.
\begin{example}\label{examplenonclosed}
Let $x=(x_n)$ be the real sequence defined by $x_1=1$ and $x_n=1/f(n)$, where $f(n)$ is the least prime factor of $n$. Fix also a real parameter $\alpha \ge -1$ and let $\mathcal{I}_\alpha$ be the ideal of subsets of zero $\alpha$-density, as defined in \eqref{eq:Ialpha}. 

It is easily seen that each $1/p$, with $p$ prime, is a $\mathcal{I}_\alpha$-limit point of $x$: indeed, if $A=\{a_n: n\in \mathbf{N}\}$ is the canonical enumeration of $p\mathbf{N}\setminus \bigcup_{q < p, q\text{ prime}}q\mathbf{N}$, then $A$ has $\alpha$-asymptotic density $\frac{1}{p}\prod_{q<p, q\text{ prime}}\left(1-\frac{1}{q}\right)>0$ and $x_n=1/p$ for all $n\in A$. On the other hand, $0\notin \Lambda_x(\mathcal{I}_\alpha)$: indeed, if a subsequence $(x_{n_k})$ converges to $0$, then for each $\varepsilon>0$ there exists a finite set $S=S(\varepsilon)$ and a prime $p=p(\varepsilon)$ such that
$$
\{n_k: k \in \mathbf{N}\}\subseteq S\cup \{n_k: |x_{n_k}|<\varepsilon\}\subseteq S\cup \{n:f(n)\ge p\}\subseteq S\cup p\mathbf{N}.
$$

Recalling that $\mathrm{d}_\alpha^\star$ is $(-1)$-homogeneous, i.e., $\mathrm{d}_\alpha^\star\,(kX)=\mathrm{d}_\alpha^\star\,(X)/k$ for all $X\subseteq \mathbf{N}$ and integers $k \ge 1$, 
(hence, in particular, strechable), 
monotone, and subadditive, cf. \cite[Example 4]{LeoTri}, it follows that 
$$
\mathrm{d}_\alpha^\star\,(\{n_k: k \in \mathbf{N}\}) \le 
\mathrm{d}_\alpha^\star\,(S\cup p\mathbf{N})\le 
\mathrm{d}_\alpha^\star\,(S)+\mathrm{d}_\alpha^\star\,(p\mathbf{N})=1/p.
$$
Since $p(\varepsilon) \to \infty$ as $\varepsilon \to 0$, then $\{n_k: k \in \mathbf{N}\} \in \mathcal{I}_\alpha$. In particular, $\Lambda_x(\mathcal{I}_\alpha)$ is not closed.
\end{example}

A stronger result holds in the case that the ideal is summable. In this regard, let $f: \mathbf{N} \to [0,\infty)$ be a function such that $\sum_{n\ge 1}f(n)=\infty$. Then, the \emph{summable ideal} generated by $f$ is
$$
\mathcal{I}_f:=\left\{S\subseteq \mathbf{N}: \sum_{n \in S}f(n)<\infty\right\}.
$$

\begin{thm}\label{thm:summableclosed}
Let $x=(x_n)$ be a sequence taking values in a metric space $X$ and let $\mathcal{I}_f$ be a summable ideal. Then $\Lambda_x(\mathcal{I}_f)$ is closed.
\end{thm}
\begin{proof}
The claim is clear if $\Lambda_x(\mathcal{I}_f)$ is empty. Hence, let us suppose hereafter that $\Lambda_x(\mathcal{I}_f) \neq \emptyset$. Let $(\ell_m)$ be a sequence in $\Lambda_x(\mathcal{I}_f)$ 
converging (in the ordinary sense) to $\ell$. 
Then, for each $m$ there exists $A_m \subseteq \mathbf{N}$ such that $\lim_{n\to \infty,\, n \in A_m} x_{n} = \ell_m$ and $A_m \notin \mathcal{I}_f$, i.e., 
$
\sum_{a \in A_m}f(a)=\infty
$. 
Let $d$ denote the metric on $X$ and, for each $m \in \mathbf{N}$, let $B_m$ be the set defined in \eqref{eq:bm}. Similarly to the proof of Theorem \ref{thm:Fsigma}, set $\theta_0:=0$ and define recursively the increasing sequence of positive integers $(\theta_m: m \in \mathbf{N})$ so that $\theta_m$ is the smallest integer greater than both $\theta_{m-1}$ and $\max(A_{m+1}\setminus B_{m+1})$ for which
$$
\sum_{a \in A_m \cap (\theta_{m-1},\theta_m]}f(a) \ge 1. 
$$

Finally, set 
$
A:=\bigcup_{m \in \mathbf{N}} A_m \cap (\theta_{m-1},\theta_m].
$ 
It follows by construction that $A\notin \mathcal{I}_f$. Moreover, for each $\varepsilon>0$, we have that $\{n \in A: d(x_n,\ell)\}$ is finite with a reasoning analogue to \eqref{eq:m0} and \eqref{eq:m1}. In particular, $\lim_{n\to \infty,\, n \in A} x_{n} = \ell$, completing the proof.
\end{proof}
 

\section{Subsequences Limit Points}

Consider the natural bijection between the collection of all subsequences of $(x_n)$ and real numbers $\omega \in (0,1]$ with non-terminating dyadic expansion 
$\sum_{i\ge 1}d_i(\omega)2^{-i}$, where we identify a subsequence $(x_{n_k})$ of $(x_n)$ with $\omega \in (0,1]$ if and only if $d_i(\omega)=1$ if $i=n_k$, for some $k \in \mathbf{N}$, and $d_i(\omega)=0$ otherwise, cf. \cite[Appendix A31]{MR1324786} and \cite{MR1260176}. Accordingly, for each $\omega \in (0,1]$, denote by $x \upharpoonright \omega$ the subsequence of $(x_n)$ obtained by omitting $x_i$ if and only if $d_i(\omega)=0$.

In addition, let $\lambda: \mathscr{M}\to \mathbf{R}$ denote the Lebesgue measure, where $\mathscr{M}$ stands for the completion of the Borel $\sigma$-algebra on $(0,1]$. 

Finally, let $\Omega$ be the set of normal numbers, i.e.,
$$
\Omega:=\left\{\omega \in (0,1]: \lim_{n\to \infty}\frac{d_1(\omega)+\cdots+d_n(\omega)}{n}=\frac{1}{2}\right\}.
$$

\begin{lem}\label{lem:inclusion}
Let $\mathcal{I}$ be a weakly thinnable ideal and let $x=(x_n)$ be a sequence taking values in a topological space. Then 
$
\lambda\left(\left\{\omega \in (0,1]: \Lambda_{x\upharpoonright \omega}(\mathcal{I} )\subseteq \Lambda_x(\mathcal{I} )\right\}\right)=1.
$
\end{lem}
\begin{proof}
It follows by Borel's normal number theorem \cite[Theorem 1.2]{MR1324786} that $\Omega \in \mathscr{M}$ and $\lambda(\Omega)=1$. Fix $\omega \in \Omega$ and denote by $(x_{n_k})$ the subsequence $x\upharpoonright \omega$. Let us suppose that $\Lambda_{x\upharpoonright \omega}(\mathcal{I} )\setminus \Lambda_x(\mathcal{I} ) \neq \emptyset$ and fix a point $\ell$ therein. Then, the set of indexes $\{n_k: k\in \mathbf{N}\}$ has asymptotic density $\nicefrac{1}{2}$ and, by hypothesis, there exists a subsequence $\left(x_{n_{k_m}}\right)$ of $(x_{n_k})$ such that $\{k_m: m \in \mathbf{N}\} \notin \mathcal{I} $ and $\lim_m x_{n_{k_m}}=\ell$. On the other hand, since $\mathcal{I} $ is weakly thinnable, the set $\{n_{k_m}: m \in \mathbf{N}\}$ does not belong to $\mathcal{I} $. Considering that $\left(x_{n_{k_m}}\right)$ is clearly a subsequence of $(x_n)$, it follows that $\ell \in \Lambda_x(\mathcal{I} )$, which contradicts our assumption. This proves that $\Lambda_{x\upharpoonright \omega}(\mathcal{I} )\subseteq \Lambda_x(\mathcal{I} )$ for all $\omega \in \Omega$.
\end{proof}

The following result is the analogue of \cite[Theorem 3.1]{Leo17} for ideal limit points:
\begin{thm}\label{thm:mainnew1}
Let $\mathcal{I}_\varphi$ be a strongly thinnable analytic P-ideal and let $x=(x_n)$ be a sequence taking values in a separable metric space. Then 
$$
\lambda\left(\left\{\omega \in (0,1]: \Lambda_x(\mathcal{I}_\varphi)=\Lambda_{x\upharpoonright \omega}(\mathcal{I}_\varphi)\right\}\right)=1.
$$
\end{thm}
\begin{proof}
Thanks to Lemma \ref{lem:inclusion}, it is sufficient to show that
\begin{equation}\label{eq:claim224}
\lambda\left(\left\{\omega \in (0,1]: \Lambda_x(\mathcal{I}_\varphi) \subseteq \Lambda_{x\upharpoonright \omega}(\mathcal{I}_\varphi)\right\}\right)=1.
\end{equation}
This is clear if $\Lambda_x(\mathcal{I}_\varphi)$ is empty. Otherwise, let us suppose hereafter that $\Lambda_x(\mathcal{I}_\varphi)\neq \emptyset$. Note that, by the $\sigma$-subadditivity of $\lambda$, Claim \eqref{eq:claim224} would follow from
\begin{equation}\label{eq:realclaim0}
\lambda\left(\left\{\omega \in (0,1]: \Lambda_x(\mathcal{I}_\varphi,q) \subseteq \Lambda_{x\upharpoonright \omega}(\mathcal{I}_\varphi)\right\}\right)=1
\end{equation}
for each (rational) $q>0$. At this point, recall from Theorem \ref{thm:Fsigma} that each $\Lambda_x(\mathcal{I}_\varphi,q)$ is closed and observe that, since $X$ is a separable metric space, every closed set is separable. 
Hence, fix a sufficiently small $q>0$ such that $\Lambda_x(\mathcal{I}_\varphi,q)\neq \emptyset$ and let $L$ be a (non-empty) countable subset with closure $\Lambda_x(\mathcal{I}_\varphi,q)$.

Fix $\ell \in L$. By hypothesis there exists a subsequence $(x_{n_k})$ such that $\lim_k x_{n_k}=\ell$ and $\|A\|_\varphi\ge q$, where $A:=\{n_k:k \in \mathbf{N}\}$. Define the set
$$
\Theta_\ell:=\left\{\omega \in (0,1]: \lim_{k\to \infty}\frac{d_{n_1}(\omega)+\cdots+d_{n_k}(\omega)}{k}=\frac{1}{2}\right\}.
$$
It follows again by Borel's normal number theorem that $\Theta_\ell \in \mathscr{M}$ and $\lambda(\Theta_\ell)=1$. 
Fix also $\omega \in \Theta_\ell$ and denote by $(x_{m_k})$ the subsequence $x\upharpoonright \omega$. Then, letting $B:=\{m_k: k \in \mathbf{N}\}$, we obtain that $A\cap B$ admits asymptotic density $\nicefrac{1}{2}$ relative to $A$, i.e., the set $K:=\{k: n_k \in B\}$ admits asymptotic density $\nicefrac{1}{2}$. Since $\mathcal{I}_\varphi$ is strongly thinnable, there exists a positive constant $\kappa=\kappa(q)$ such that 
$$
\|A_K\|_\varphi=\|A \cap B\|_\varphi \ge \kappa q.
$$
In addition, since $C:=\{k: m_k \in A_K\} \le A_K$, we get by the strongly thinnability of $\mathcal{I}_\varphi$ that $\|C\|_\varphi \ge cq$, for some $c>0$. It follows by construction that the subsequence $(x_{m_k}: k \in C)$ of $(x_{m_k}: k \in \mathbf{N})$ converges to $\ell$, hence $\ell \in \Lambda_{x\upharpoonright \omega}(\mathcal{I}_\varphi,cq)$ for all $\omega \in \Theta_\ell$.

Thus, define $\Theta:=\bigcap_{\ell \in L}\Theta_\ell$ and note that $\Theta \in \mathscr{M}$ and $\lambda(\Theta)=1$. Therefore  
$
\lambda\left(\left\{\omega \in \Theta: L \subseteq \Lambda_{x\upharpoonright \omega}(\mathcal{I}_\varphi,cq)\right\}\right)=1.
$ 
On the other hand, each $\Lambda_{x\upharpoonright \omega}(\mathcal{I}_\varphi,cq)$ is closed by Theorem \ref{thm:Fsigma}, hence it contains the closure of $L$, that is,  
$$
\lambda\left(\left\{\omega \in \Theta: \Lambda_x(\mathcal{I}_\varphi,q) \subseteq \Lambda_{x\upharpoonright \omega}(\mathcal{I}_\varphi,cq)\right\}\right)=1.
$$
This implies \eqref{eq:realclaim0} since $\Lambda_{x\upharpoonright \omega}(\mathcal{I}_\varphi,cq) \subseteq \Lambda_{x\upharpoonright \omega}(\mathcal{I}_\varphi)$, completing the proof.
\end{proof}

As a consequence of Corollary \ref{cor:Ialpha} and Theorem \ref{thm:mainnew1}, we obtain:
\begin{cor}\label{stat}
Let $x$ be a sequence taking values in a separable metric space. Then the set of statistical limit point of $x$ is equal to the set of statistical limit points of almost all its subsequences \textup{(}in the sense of Lebesgue measure\textup{)}.
\end{cor}

With a similar argument, the following can be shown (we omit details):
\begin{thm}\label{thm:mainnew2}
Let $\mathcal{I}_f$ be a thinnable summable ideal and $(x_n)$ be a sequence taking values in a separable metric space $X$. Then 
$$
\lambda\left(\left\{\omega \in (0,1]: \Lambda_x(\mathcal{I}_f)=\Lambda_{x\upharpoonright \omega}(\mathcal{I}_f)\right\}\right)=1.
$$
\end{thm}

We conclude with the relationship between ideal limit points and ideal cluster points of subsequences of a given sequence. Given an ideal $\mathcal{I}$ and a sequence $x=(x_n)$ taking values in a topological space, recall that $\ell$ is a $\mathcal{I}$-cluster point of $(x_n)$ provided that $\{n: x_n \in U\}\notin \mathcal{I}$ for all neighborhoods $U$ of $\ell$. Denoting by $\Gamma_x(\mathcal{I})$ the set of $\mathcal{I}$-cluster points of $(x_n)$, we obtain:
\begin{cor}\label{cor:final}
Let $x$ be a sequence taking values in a separable metric space and let $\mathcal{I}$ be a thinnable summable ideal or a strongly thinnable analytic P-ideal. Then 
$$
\lambda\left(\left\{\omega \in (0,1]: \Lambda_{x\upharpoonright \omega}(\mathcal{I})=\Gamma_{x\upharpoonright \omega}(\mathcal{I})\right\}\right)
$$
is either $0$ or $1$. In addition, it is $1$ if and only if $\Lambda_{x}(\mathcal{I})=\Gamma_{x}(\mathcal{I})$.
\end{cor}
\begin{proof}
Thanks to Theorem \ref{thm:mainnew1}, Theorem \ref{thm:mainnew2}, and \cite[Theorem 3.1]{Leo17}, it holds
$$
\lambda\left(\left\{\omega \in (0,1]: \Lambda_{x\upharpoonright \omega}(\mathcal{I})=\Lambda_x(\mathcal{I}) \,\text{ and }\,\Gamma_{x\upharpoonright \omega}(\mathcal{I})=\Gamma_x(\mathcal{I})\right\}\right)=1,
$$
which is sufficient to prove the claim.
\end{proof}

\subsection*{Acknowledgments} 
The author is grateful to Marek Balcerzak (\L{}\'{o}d\'{z} University of Technology, PL) and the anonymous referee for several useful comments.

\subsection*{Note added in proof} Theorems \ref{thm:Fsigma} and \ref{thm:summableclosed} hold also in first countable Hausdorff spaces, see \cite[Section 2]{PaoloMarek17}. Moreover, it turns out that the topological analogue of Theorem \ref{thm:mainnew1} is quite different, providing a non-analogue between measure and category. Indeed, it has been shown in \cite{Leo17c}, cf. also \cite{LMM}, that, if $x$ is a sequence in a separable metric space, then $\{\omega \in (0,1]: \Lambda_x(\mathcal{I}_0)=\Lambda_{x \upharpoonright \omega}(\mathcal{I}_0)\}$ is not a first Baire category set if and only if every ordinary limit point of $x$ is also a statistical limit point of $x$.

\bibliographystyle{amsplain}
\bibliography{ideale}

\end{document}